\documentclass{amsart}
\usepackage{amsmath,amsthm,amssymb}
\usepackage{graphicx}
\usepackage{enumerate}
\usepackage{subfigure}

\newcommand{\+}{\oplus}
\renewcommand{\.}{\odot}
\renewcommand{\/}{\oslash}
\newcommand{\0}{0_\T}
\newcommand{\1}{1_\T}

\newcommand{\CP}{\mathbb{CP}}
\newcommand{\C}{\mathbb{C}}
\newcommand{\N}{\mathbb{N}}
\newcommand{\R}{\mathbb{R}}
\newcommand{\T}{\mathbb{T}}
\newcommand{\TT}{\mathcal{T}}
\renewcommand{\AA}{\mathcal{A}}
\newcommand{\Z}{\mathbb{Z}}
\newcommand{\w}{\omega}
\renewcommand{\L}{\mathcal{L}}

\DeclareMathOperator{\val}{val}
\DeclareMathOperator{\Val}{Val}
\newtheorem{thm}{Theorem}[section]
\newtheorem{lem}[thm]{Lemma}
\newtheorem{cor}[thm]{Corollary}

\theoremstyle{definition}
\newtheorem{defn}{Definition}[section]
\newtheorem{eg}{Example}[section]
\theoremstyle{remark}
\newtheorem{rmk}{Remark}[section]

\title{Working with Tropical Meromorphic Functions of One Variable}


\author{Yen-lung Tsai}
\address{Yen-lung Tsai\\
{\tt {yenlung@nccu.edu.tw}}\\
Department of Mathematical Sciences\\
National Chengchi University\\
NO.64, Sec.2, ZhiNan Rd.,Wenshan District,
Taipei 11605, Taiwan}
\thanks{This research was supported by the National Science Council (NSC 99-2115-M-004-003-, and 98-2115-M-004-002-).}
\begin{document}
\maketitle
\begin{abstract}
In this paper, we survey and study definitions and properties of tropical polynomials, tropical rational functions and in general, tropical meromorphic functions, emphasizing practical techniques that can really carry out computations. For instance, we introduce maximally represented tropical polynomials and tropical polynomials in compact forms to quickly find roots of given tropical polynomials. We also prove the existence and uniqueness of tropical theorems for meromorphic functions with prescribed roots and poles.  Moreover, we explain the relations between classical and tropical meromorphic functions. Different definitions and applications of tropical meromorphic functions are discussed. Finally, we point out the properties of tropical meromorphic functions are very similar to complex ones and prove some tropical analogues of theorems in complex analysis.
\end{abstract}

\section{Introduction}
Tropical geometry has been rapid developed in recent years. Roughly speaking, tropical geometry study the image of classical geometric objects through a certain valuation map. Therefore, one can expect the properties we find in tropical geometry somehow reflect the properties in classical geometry. Indeed, many important applications has been carry out. Among many others, Mikhalkin~\cite{mikhalkin03, mikhalkin05} calculated Gromov-Witten invariants in $\CP^2$; Itenberg, Kharlamov, and Shustin~\cite{iks03} calculated Welschinger invariants. Besides, there are also many mathematicians try to understand mirror symmetry through tropical geometry~\cite{g10}.

On the other hand, since tropical geometry is naturally related to combinatorics analysis and convex optimization. It provides lots of possible applications to real world also. For instance, Pacher and Sturmfels~\cite{ps04} applied it to statistics. As we will see, the graph of a tropical meromorphic function is piecewise linear. This kind of functions arise in real world and with tropical geometry, we could study these functions directly with some standard analysis techniques. For example, we can apply techniques of Nevanlinna theory to solve difference equations~\cite{hs09, lt09}.

Though with so many interesting applications, the foundation of tropical geometry is far from complete, and some  are known to experts but hardly find in single monograph. In this paper, we survey and study very basic elements in tropical geometry, namely tropical polynomials, rational functions, and meromorphic functions. We focus on the concepts of the definitions, and the techniques that one can really compute. In short, we would like to provide necessary tools to work with tropical meromorphic functions.

One important and interesting thing we want to point out is that the tropical mathematics though work on real numbers, but indeed also reflect the world of complex geometry. Many properties of tropical meromorphic functions act like complex meromorphic functions. For instance, we have tropical version of the Fundamental Theorem of Algebra, Liouville's Theorem, and (modified) Maximum Modulus Theorem.

One problem to deal with tropical meromorphic functions is that there are some subtle different definitions in different contexts. The most natural choice of definition is that $-\infty$ (the tropical zero) can be a root or pole of some tropical meromorphic functions, and those are the major definition of tropical meromorphic functions for us. Howerver, sometimes it is convenient to work with tropical meromorphic functions defining on $\R$ (no $-\infty$) only, and we will call these $\R$-tropical meromorphic functions. Moreover, tropical meromorphic functions are piecewise linear functions of integer slopes. We can drop the integer slope assupmtion, allowing real slopes, and we will call the new type of functions the extended tropical meromorphic functions.

\section{Tropical Polynomials}

\begin{defn} [Tropical Semiring]\label{semiring}
Let $\T = \R \cup \{ -\infty \}$. The tropical semiring $(\T, \+, \.)$ is an algebraic structure with two binary operators defined as followings.
\[
\begin{cases}
a \+ b = \max \{a,b\},\\
a \. b = a + b.
\end{cases}
\]
\end{defn}

We can easily find out that the additional identity of the tropical semiring is $\0 = -\infty$ and the multiplicity identity is $\1 = 0$. There is no subtraction in tropical semiring. For example, the equation $2 \+ k = 3$ yields no solution for $k$. Strangely enough, since the tropical multiplication is just the original addition, so we do have division in tropical semiring, namely $x \/ y = x - y$.

One can actually define the tropical semiring in different ways. For example, the addition of the tropical semiring can be defined as $a \+ b = \min \{a, b\}$. We call the semiring the min-plus tropical semiring. We will focus on the semiring as in the Definition~\ref{semiring}, which we call the max-plus tropical semiring, and we usually just call it the tropical semiring.

Naturally, we will define the tropical polynomial as a polynomial with coefficients in $\T$ and computed by the tropical addition and the tropical multiplication.

\begin{defn}[Tropical Polynomial]
A \emph{tropical polynomial} $f(x)$ is of the form 
\[
a_n \. x^{\. n} \+ a_{n-1} \. x^{\. (n-1)} \+ \cdots \+ a_0,
\]
where $n$ is a positive integer, and $a_0, a_1, \ldots, a_n \in \T$. Evaluate $f(x)$, we obtain
\[
f(x) = \max \{ nx + a_n, (n-1)x + a_{n-1},  \ldots, a_0 \}
\]
We usually omit the terms with the coefficients $-\infty$.
\end{defn}

\begin{rmk}
As in classical polynomials, a tropical monomial $x$ means $1_\T \cdot x$, but $1_\T = 0$, so $x = 0 \. x$. Moreover, a polynomial $x \. (x \+ 2) = (x \+ -\infty) \. (x \+ 2)$.
\end{rmk}

We now give two examples to elaborate how do we define roots of a tropical polynomial and the multiplicities of these roots.

\begin{eg}
Let $f(x) = x^{\. 2} \+ 2\. x \+ 3$. Figure~\ref{fig:poly00} shows the graph of the function. Note that $f(x) = (x \+ 1) \. (x \+ 2)$ (as a function), so we can think the points $x=1$ and $x=2$ are the ``roots'' of $f(x)$. These are exactly the points where the function $f$ is not differentiable. On the other hand, $f(x) = \max \{ 2x, 2+x, 3\}$ by definition. Note that $x=1$ and $x=2$ are the points where the maximum achieves at least twice in these linear terms.

As in classical cases, we should call $x=1$ and $x=2$ the roots of polynomial $f(x)$, each with multiplicities $1$. At the point $x=1$, we calculate change of the slopes at that point, we have $\lim_{x \to -1^+} f'(x) - \lim_{x \to -1^-} f'(x) = 1$. Similarly, calculate change of the slopes of the graph at $x=2$, we get $1$, which coincide with the multiplicities we expect of these points. 
\begin{figure}[h!]
\begin{center}
\includegraphics[scale=0.7]{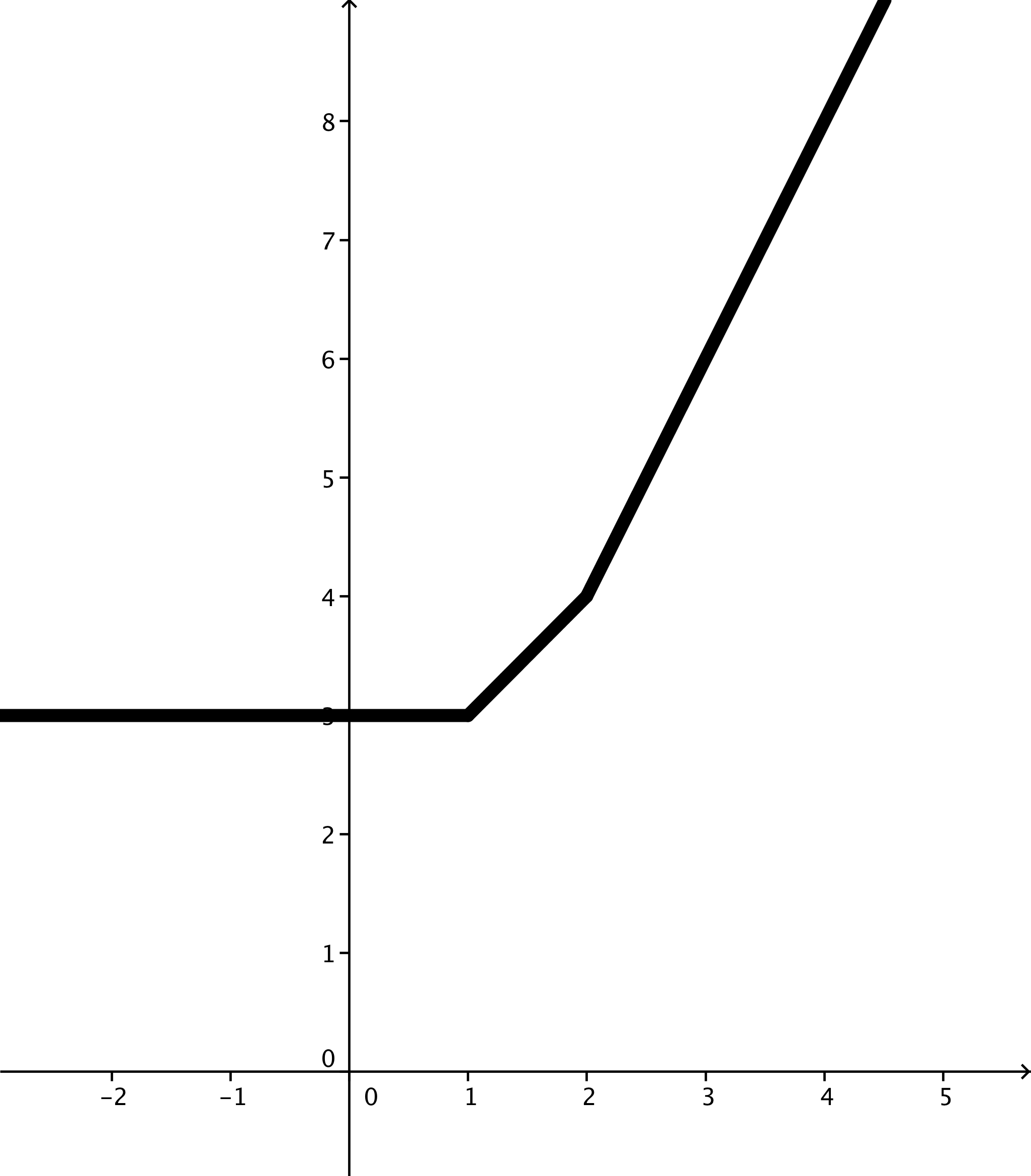}
\caption{The graph of $f(x) = x^{\. 2} \+ 2\. x \+ 3$.}\label{fig:poly00}
\end{center}
\end{figure}
\end{eg}

\begin{eg}
Let $f(x) = (x \+ 1)^{\. 2}$. Figure~\ref{fig:poly01} shows the graph of the function. At the point $x=1$, the change of slopes is $2$, which again coincide with the multiplicity we expect. 
\begin{figure}[h!]
\begin{center}
\includegraphics[scale=0.7]{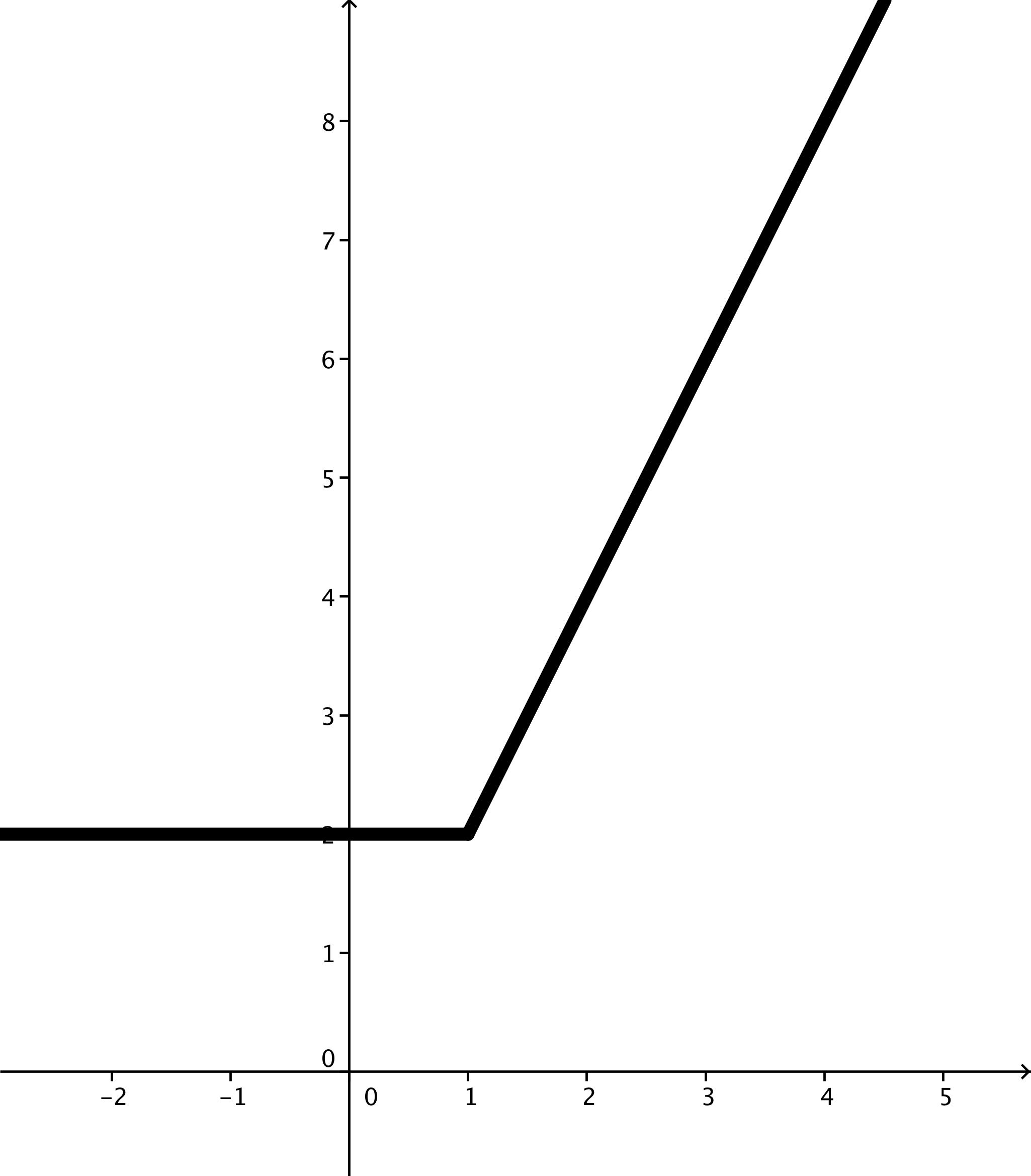}
\caption{The graph of $f(x) = (x \+ 1)^{\. 2}$.}\label{fig:poly01}
\end{center}
\end{figure}
\end{eg}

The examples motivate the following definitions, we basically follow the definitions in~\cite{hs09, lt09}.

\begin{defn}\label{defn:rmero}
A \emph{tropical meromorphic function} $f\colon \R \to \R$ is a continuous piecewise linear function on $\R$ if both one-sided derivatives are integers at each point $x \in \R$. 
\end{defn}

In this paper, if $f(x)$ is a tropical meromorphic function defined on $\R$, we call $f(x)$ an \emph{$\R$-tropical meromorphic function}. The multiplicity of a root is the change of slopes. 

\begin{defn}\label{defn:mult}
Let $f(x)$ be an $\R$-tropical meromorphic function. Let 
\[
\w_f(x) = \lim_{\varepsilon \to 0^+} \left[ f'(x+\varepsilon) - f'(x-\varepsilon) \right].
\]
\begin{enumerate}[(a)]
\item If $\w_f(x) > 0$, we call $x$ a \emph{root} of $f(x)$ with multiplicity $\w_f(x)$.
\item If $\w_f(x) < 0$, we call $x$ a \emph{pole} of $f(x)$ with multiplicity $-\w_f(x)$.
\end{enumerate}
\end{defn}

We can also define a tropical meromorphic function on $\T$.

\begin{defn}
We say that a function $f$ is a \emph{tropical meromorphic function on $\T$} if 
\begin{enumerate}[(a)]
\item $f$ is a tropical meromorphic function on $\R$, and
\item there exist $x_0 \in \R$ such that $f'(x_0)$ is constant for all $x < x_0$.
\end{enumerate}
\end{defn}

\begin{defn}
Let $f(x)$ be a tropical meromorphic function on $\T$. There is a number $x_0 \in \R, m \in \Z$ such that $f'(x) = m$ for all $x < x_0$.
\begin{enumerate}[(a)]
\item We say that $-\infty$ is a pole of $f$ if $m < 0$.
\item We say that $-\infty$ is a root of $f$ if $m > 0$.
\item We say that $-\infty$ is an ordinary point if $m=0$.
\end{enumerate}
\end{defn}

We will mainly discuss the tropical meromorphic functions on $\T$. Unless stated otherwise, when we say $f(x)$ is a tropical meromorphic (polynomial, rational) function, we mean that $f(x)$ is a tropical meromorphic (polynomial, rational) function on $\T$.

\begin{defn}\label{defn:rational}
Let $f(x)$ be a tropical rational function if there are two tropical polynomials $g(x)$, $h(x)$ such that 
\[
f(x) = h(x)\/g(x).
\]
\end{defn}

We will see these definitions are reasonable. For instance, in Section~\ref{s:prescribed}, we will show that a tropical meromorphic function $f(x)$ is a tropical rational function if and only if $f(x)$ has finitely many roots and poles.
\section{Maximally Represented Tropical Polynomials}
For complex or real cases, different polynomials define different functions. However, different tropical polynomials might define the same function as the following example shows. 

\begin{eg}\label{eg:differentforms}
As in Figure~\ref{fig:samegraph}, the graphs of $f(x) = (-1)\.x^{\. 2}  \+  x \+ 1$ and $g(x) = (-1) \. x^{\. 2}  \+ (-1) \. x \+ 1$ are the same, so polynomials $f(x)$ and $g(x)$ define the same function.
\end{eg}

\begin{figure}[h]
\begin{center}
\includegraphics[scale=1]{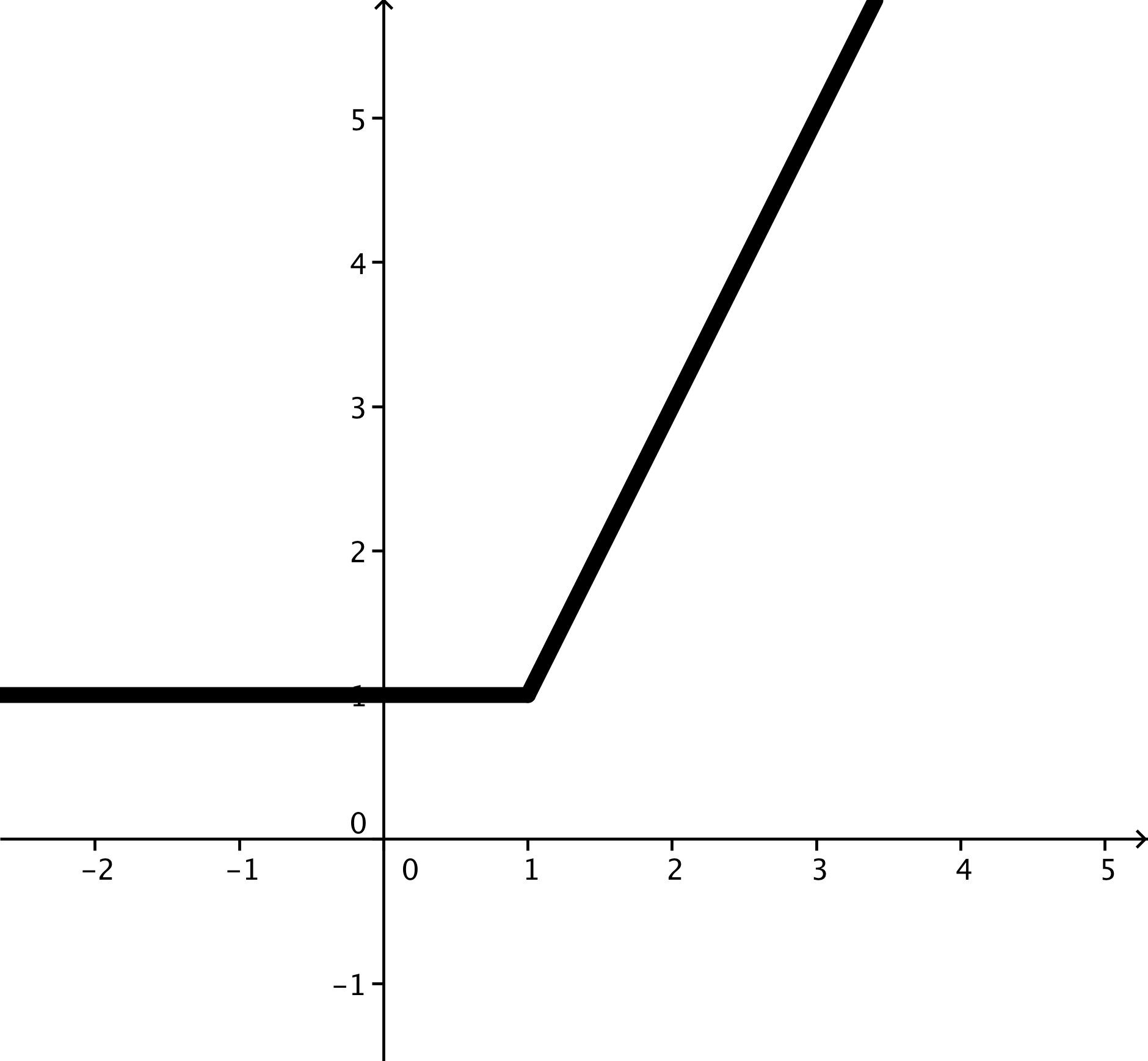}
\caption{The graphs of two different polynomials $f(x) =(-1)\.x^{\. 2}  \+  x \+ 1$ and $g(x) = (-1) \. x^{\. 2}  \+ (-1) \. x \+ 1$ are the same.}\label{fig:samegraph}
\end{center}
\end{figure}

\begin{defn}
Let $f(x)$ and $g(x)$ be two tropical polynomials. We say that $f(x)$ and $g(x)$ are \emph{equivalent} if $f(x)$ and $g(x)$ define the same function.
\end{defn}

Though different tropical polynomials might define the same function, we can choose a ``good'' representation out of a collection of equivalent tropical polynomials. As we will see, every tropical polynomial has a unique \emph{maximal representation}. Different maximally represented polynomials define different functions. We can use the Legendre transformation to get the desire maximal representation of a tropical polynomial in question.

First, we observe that any one-variable tropical polynomial defines a real convex function.
\begin{thm}
Let $f(x)$ be a tropical polynomial. Let $h\colon \R \to \R$ be the function defined by
\[
h = f|_\R.
\]
Then $h(x)$ is a convex function.
\end{thm}
\begin{proof}
Obviously, $h(x)$ is a continuous, piecewise linear function. Therefore, we just need to show that for each point, the change of slopes is either zero or positive. Now, if for any real number $x_0$, if in a neighborhood of $x_0$ that $h(x)$ is linear, then the change of slope is zero. If $x_0$ is a point that $h$ fails to be linear, then in a small neighborhood of $x_0$, $h(x)$ can be represented by the maximum of two linear functions, say, $i_1 x + a_1$ and $i_2 x + a_2$, where $i_2 > i_1$. Since $i_2 x_0 + a_2 = i_1 x_0 + a_1$. $i_2 (x_0 + \varepsilon) + a_2 > i_1 (x_0 + \varepsilon) + a_1$ for all $\varepsilon > 0$. Hence, 
\[
\lim_{\varepsilon \to 0^+} h(x_0 + \varepsilon) - h(x_0 - \varepsilon) = i_2 - i_1 > 0.
\]
Thus we conclude that the change of slope at any real number for the function $h(x)$ is either zero or positive.
\end{proof}

Let $h \colon \R^n \to \R$ be a convex function. The Legendre transform $\L_h \colon \R^n \to \R$ of $h(x)$ is defined by 
\[
\L_h (\alpha) = \max_{x \in \R^n} \{\alpha \cdot x - h(x)\}
\]

For the purpose of this paper, we consider that $h \colon \R \to \R$ is induced by a one-variable tropical polynomial $f(x) = a_n\. x^n \+ \cdots \+ a_r\. x^r$. Then $h(x)$ is just the function defined by $f(x)$ with restriction on $\R$. Note that $\L_h$ is just a partially defined function. For instance, let $h$ be the function induced by the tropical polynomial $x+2$, then the value of 
\begin{align*}
\L_h(2) 	&= \max\{2x - h(x)\} \\
			&= \max\{2x - \max\{x,2\}\}
\end{align*}
tends to infinity.

We modify the definition to suit our needs, when we say the Legendre transform, we really mean the modified one for the rest of the paper.

\begin{defn}[Modified Legendre Transform]\label{def:legendre}
Let $f(x) = a_n \. x^n \+ \cdots \+ a_r \. x^r$, be a one variable tropical polynomial, where $n, n-1, \ldots, r$ are integers such that $n \leq r$.  The modified \emph{Legendre transform} $\L_f$  corresponding to $f(x)$ is a real valued function defined on $\{r, r+1, \ldots, n\}$ such that
\[
\L_f(p) = \max_{x \in \R} \{p\cdot x - f(x)\}
\]
\end{defn}

It is easy to check that the modified Legendre transform we have is well-defined. That is, for all $p$ in the domain of the function, $\L_f(p)$ is a real number.

\begin{defn}\label{def:maxpoly}
A tropical polynomial $f(x) = a_n \. x^n \+ \cdots \+ a_r \. x^r$ is \emph{maximally represented} if
\[
a_i = -\L (i),
\]
for all $i \in \{r, r+1, \ldots, n\}$.
\end{defn}

It will be clear latter why we call these polynomials maximally represented. Using the Legendre transform to find the we call maximally represented polynomials is briefly discussed in~\cite{ims09, mikhalkin06}. The detailed properties and proof of min-plus tropical algebra version can be founded in~\cite{gm07}, where the polynomials corresponding to our maximally represented polynomials analogues in min-plus tropical algebra are called the least-coefficient polynomials.

Note that for an arbitrary tropical polynomial $f(x)$, we evaluate the function $f(x)$ with sufficiently large $x$, then the maximum of $\{ a_i + i x \}_{r \leq i \leq n} $ will be $a_n + nx$.
\begin{lem}\label{lem:m}
Let $f(x) = a_n \. x^n \+ \cdots \+ a_r \. x^r$ be a tropical polynomial.
\begin{enumerate}[(a)]
\item There is a number $M > 0$ such that for any $x > M$, $f(x) = a_n + nx$. 
\item There is a number $m<0$ such that for any $x<m$, $f(x) = a_r + rx$.
\end{enumerate}
\end{lem}
\begin{proof}
We prove the part (a) and skip the proof for part (b), since the proofs are very similar.
Let $M_k = (a_k - a_n)/(n-k)$, for some integer $k$ such that $r \leq k \leq n-1$.  If $x > M_k = \frac{a_k - a_n}{n-k}$ $(n-k)x > a_k - a_n$. Hence, $a_n + nx > a_k +kx$. Take
\[
M = \max_{r \leq k \leq n-1} \{ M_k\},
\]
then for all $x > M$, we have $a_n + nx > a_k +kx$ for all $r \leq k \leq n-1$. Hence,
\[
\max_{r \leq i \leq n} \{ a_i + ix \} = a_n + nx,
\]
for all $x > M$.
\end{proof}
\begin{lem}
Let $f(x) = a_n \. x^n \+ \cdots \+ a_r \. x^r$ be a tropical polynomial. Let $g(x)$ be the polynomial obtained from $f(x)$ by adding one term $\alpha \. x^k$ such that $k>n$. Then $f(x)$ and $g(x)$ define different functions.
\end{lem}

\begin{proof}
By the Lemma~\ref{lem:m}, there exists a positive real number $M_1$ such that $f(x) = a_n + nx$ for $x > M_1$, and a positive real number $M_2$ such that $g(x) = \alpha + kx$ for $x > M_2$. Take a number $x_0 > \max\{M_1, M_2\}$ and $x_0 \neq (a_n - \alpha)/(k - n)$. Then $f(x_0) = a_n + n x_0 \neq \alpha + kx_0 = g(x_0)$.
\end{proof}
\begin{lem}\label{lem:leading}
Let $f(x) = a_n \. x^n \+ \cdots \+ a_r \. x^r$ be a tropical polynomial.
\begin{enumerate}[(a)]
\item Let $g(x)$ be a tropical polynomial obtained from $f(x)$ by substituting $a_n$ by a number $\alpha$, such that $\alpha > a_n$. Then $f(x), g(x)$ define different functions.
\item Let $g(x)$ be a tropical polynomial obtained from $f(x)$ by substituting $a_r$ by a number $\beta$, such that $\beta > a_r$. Then $f(x), g(x)$ define different functions.
\end{enumerate}
\end{lem}
Now, we want to point out an important fact. That is, any tropical polynomial is equivalent to a tropical polynomial of the form:
\[
g(x) = a_n \. x^n \+ \cdots \+ a_r \.x^r, 
\]
where $a_i \in \R$ for all $i = r, r+1, \ldots, n$. That is, there is no missing term between $x^n$ and $x^r$. The idea of the proof is as followings. The graph of a tropical polynomial is convex piecewise linear. Suppose the polynomial $f(x)$ misses the $x^k$ term, where $r < k < n$. That is, the coefficient $a_k = -\infty$. Add the $x^k$ term to the polynomial means add a line with slope $k$. If the line is below the original graph, it would not change the function. As in the Figure~\ref{fig:addmissing} shows, we can move the new line up to touch a corner of the graph of the tropical polynomial.

\begin{figure}[h]
\begin{center}
\subfigure[The added line is lower then the graph, which is okay.]{\label{fig:addmissing-a}\includegraphics[scale=0.3]{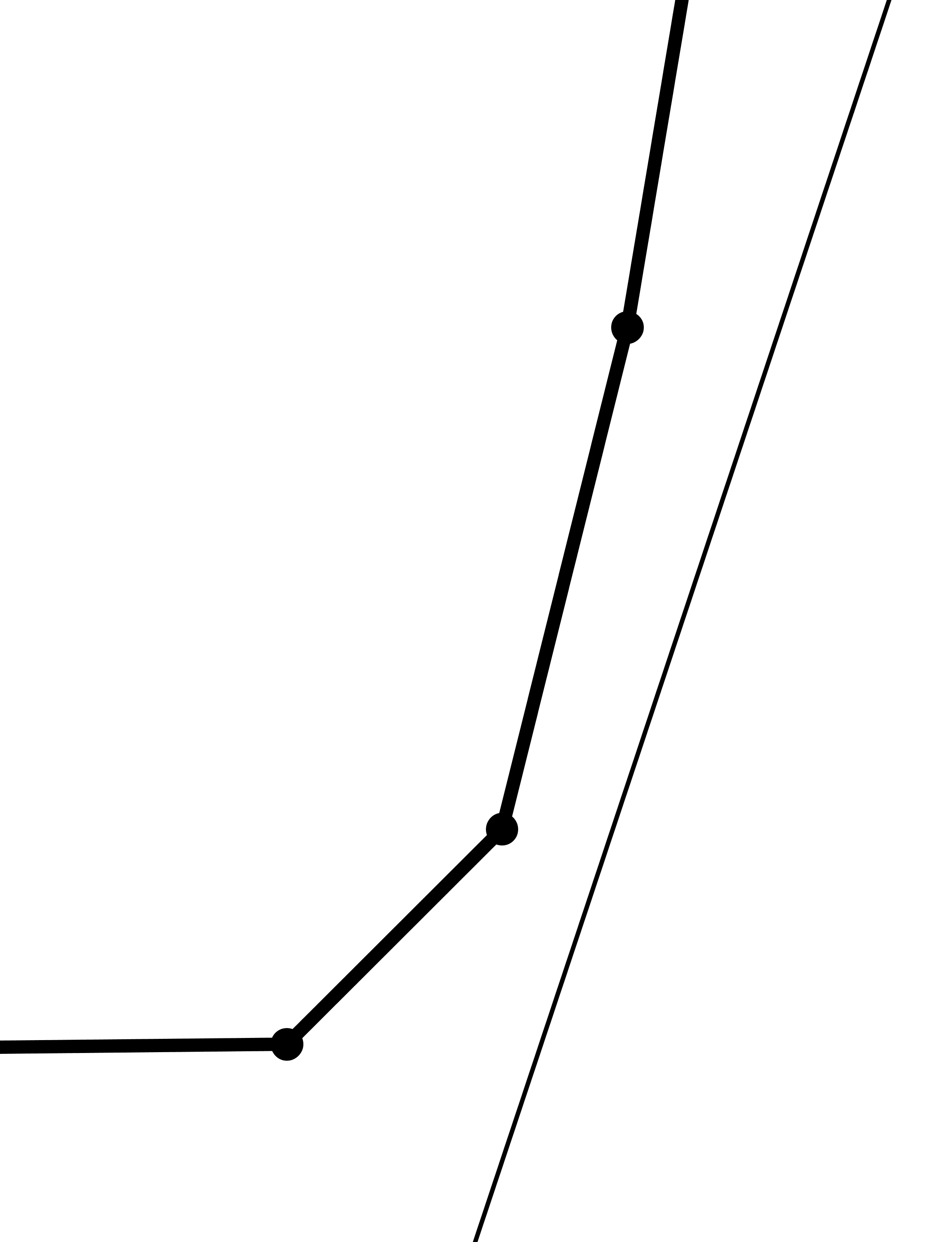}}
\subfigure[The added line up to the highest position.]{\label{fig:addmissing-b}\includegraphics[scale=0.3]{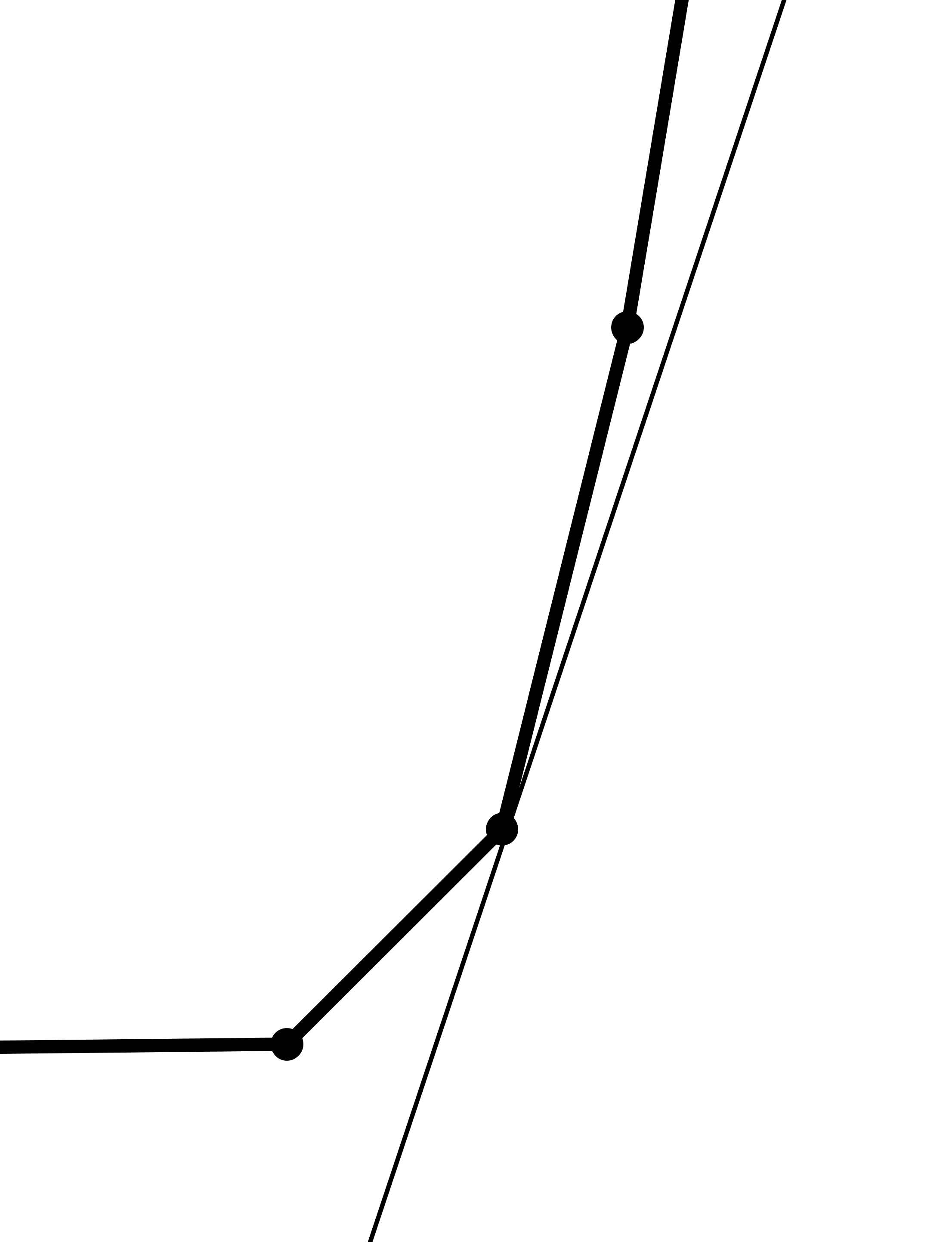}}
\subfigure[The added line is too high.]{\label{fig:addmissing-c}\includegraphics[scale=0.3]{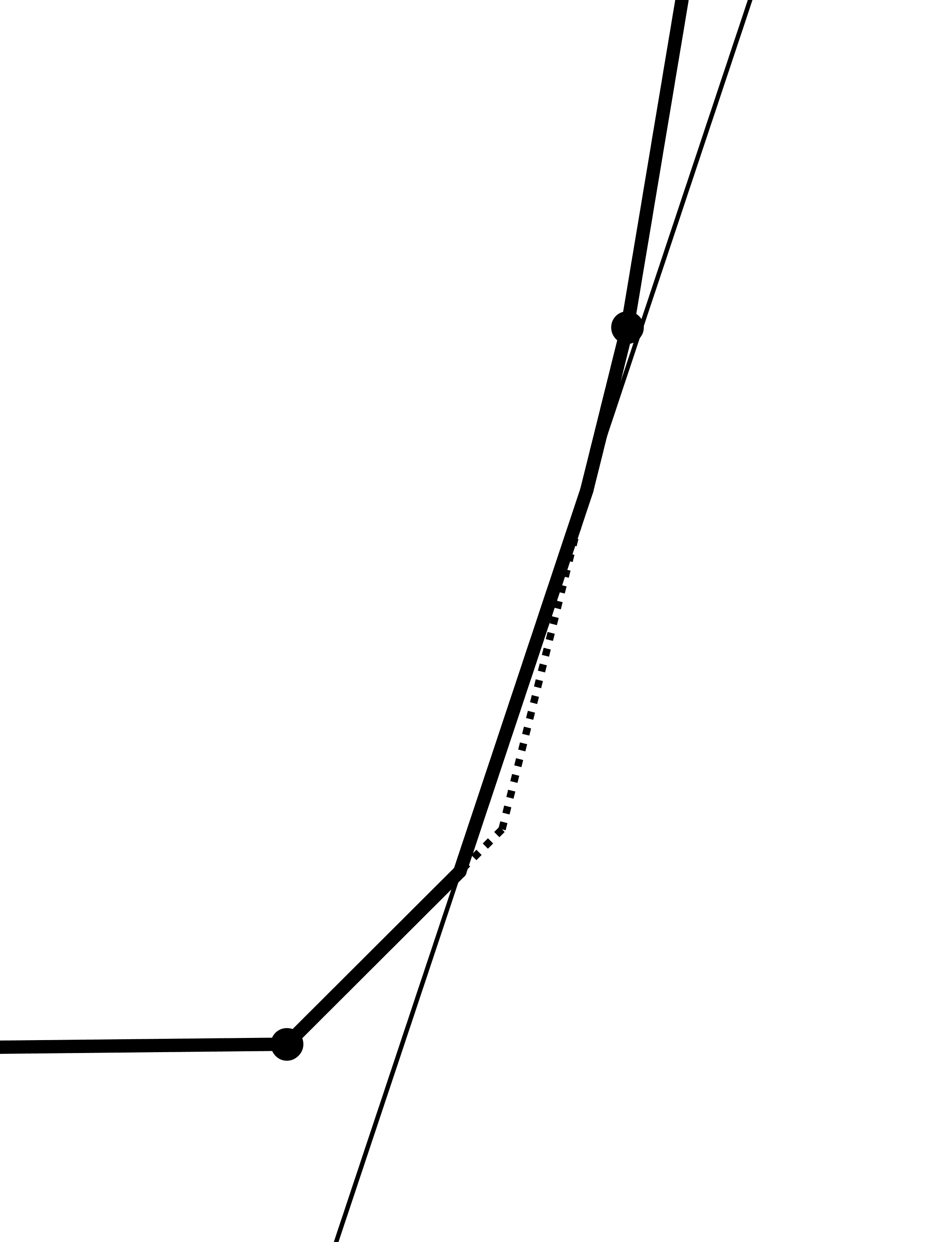}}
\caption{Adding a ``missing term'' to the graph of a tropical polynomial.}\label{fig:addmissing}
\end{center}
\end{figure}

As a result, with some real coefficients we can always add a missing $x^k$ term without change the original function. Therefore, the following theorem is obvious.

\begin{lem}\label{lem:nomissing}
Let $f(x) = a_n \. x^{\. n} \+ \cdots \+ a_r \. x^{\. r}$ be a tropical polynomial. Suppose that $a_j = -\infty$ for some $n < j < r$. Then there is $\alpha \in \R$ such that $g(x) = a_n \. x^{\. n} \+ \cdots \+ a_{j-1} \. x^{\.(j-1)} \+ \alpha \. x^j \+ a_{j+1} \. x^{j+1} \cdots \+ x^r$ and $f(x)$ define the same function..
\end{lem}
We can actually do better: find the largest possible coefficient for a missing $x^k$-term without change the original function. It turns out that the coefficient is exactly $-\L(k)$. Observe that the Legendre transform basically compare the line $y = kx$ and the graph of the tropical polynomial $f(x)$. If $y = kx$ goes to high, we need to take it down by $\L(k)$ to keep the graph the same as before. On the other hand, if $y = kx$ is too low, we need to move it up by $-\L(k)$. Therefore, the highest possible line with the slope $k$ without change the original function is $kx - \L(k)$, and it is exactly $-\L(k) \. x^{\. k}$
\begin{thm}
Let $f(x) = a_n \. x^n \+ \cdots \+ a_r \. x^r$ be a maximally represented polynomial if $g(x)$ is obtained from $f(x)$ by substitute any coefficient with a larger real number, then the graph of $g(x)$ and $f(x)$ are different. Moreover, any tropical polynomial has a unique maximally represented polynomial. 
\end{thm}
From the discussion above, we can easily find out that two tropical polynomials have the same graph if and only if their maximally represented polynomials are the same.
\begin{cor}
 Two tropical polynomial $f(x)$ and $g(x)$ define the same function if and only if their maximally represented polynomials are the same.
\end{cor}

The following Corollary is obvious, too.

\begin{cor}
Let $f(x) = a_n \. x^n \+ \cdots \+ a_r \. x^r$ be a tropical polynomial. Then $f(x)$ is maximally represented if and only if 
for each $r \leq i \leq n$, there exists some $x_0 \in \R$ such that $f(x_0) = a_i x_0^i$.
\end{cor}

Let $f(x) = a_n \. x^n \+ a_{n-1} \. x^{\. (n-1)} \+ \cdots \+ a_r \. x^r$ be a tropical polynomial. Suppose we want to find the maximally represented coefficient for $a_k$, where $r < k < n$. There are two cases to consider.

\begin{enumerate}[(i)]
\item A piece of the line $a_k + k x$ appears in the graph of $f(x)$. In this case, $a_k$ is maximally represented already.
\item The line $a_k + kx$ is under the graph of $f(x)$. In this case, we can move the line $a_k + kx$ up to intersect the graph of $f(x)$ at exactly one point. This point will be the intersection of two lines defined by two other terms $a_i \. x^{\. i}$ and $a_j \. x^{\. j}$ appear in the polynomial. Observe that one of these two lines must with slope less than $k$ and the other with slope greater than $k$. Suppose $i < k < j$, the intersection of $a_i + ix$ and $a_j + jx$ is $(a_i - a_j) / (j-i)$. If the line $\alpha + kx$ just touch the intersection, then we should have
\[
\frac{\alpha - a_j}{j - k} = \frac{a_i - a_j}{j-i}.
\]
Solve for $\alpha$, we get
\[
\alpha = \frac{(a_i - a_j) (j-k)}{j-i} + a_j.
\]
The maximum possible $\alpha$ is the maximally represented coefficient for $a_k$.
\end{enumerate}

%%%%
\section{Tropical Fundamental Theorem of Algebra}
In previous section, we prove that each tropical polynomial $f(x)$ has a unique maximally represented tropical polynomial $g(x)$. Once we get the maximally represented polynomial of $f(x)$, it is easy to find all roots of $f(x)$. We elaborate the idea here. Let
\[
f(x) = a_n \. x^{\. n} \+ a_{n-1} \. x^{\. (n -1)} \+ \cdots \+ a_r \. x^r
\]
be a maximally represented tropical polynomial. As discussed in previous section, each line $a_k + kx$ appears in certain corner of the graph of $f(x)$. It turns out that there is $x_0$ in $\R$ such that
\[
a_k + k x_0 = a_{k-1} + (k-1)x_0
\]
for all $k = r+1, r+2, \ldots, n$, and $f(x_0) =  a_k + k x_0$. Therefore, $x_0 = a_{k-1} - a_k$ is a zero of the polynomial $f(x)$. Let
\[
d_k = a_{k-1} - a_k
\]
for $k = r+1, r+2, \ldots, n$. Then $d_{r+1}, d_{r+2}, \ldots, d_n$ (not necessary distinct) are $n-r$l roots of $f(x)$.

Note that
\[
f(x) = a_n \. x^{\. r} \.  [x^{\. (n-r)} \+ (a_{n-1} - a_n) \. x^{\. (n-r-1)}\+ \cdots \+ (a_r - a_n)],
\]
so $f(x)$ has a zero at $-\infty$ with multiplicity $r$. Hence, we have exactly $n$ roots (counting multiplicities) for the polynomial $f(x)$.

It is reasonable to have the following theorem, which we call the Tropical Fundamental Theorem of Algebra.

% criteria of a poly been maximally represented
%\begin{lem}
%Let $ f(x)=a_{n}x^n\oplus\ a_{n-1}x^{n-1}\oplus \cdots \oplus\ a_{r}x^r $, where $ a_{i}\neq\ -\infty\ , \\ \forall\ i=r,r+1,\cdots\ ,n-1,n $. Then $ f(x) $ is a largest coefficient polynomial $ \Leftrightarrow\ d_{n}\geq\ d_{n-1}\geq \cdots\geq\ d_{r+1} $, where $ d_{i}=a_{i-1}-a_{i} $.
%\end{lem} 

\begin{thm}[Tropical Fundamental Theorem of Algebra]\label{thm:tfta}
Let
\[
f(x) = a_n \. x^{\. n} \+ a_{n-1} \. x^{\. (n -1)} \+ \cdots \+ a_r \. x^r
\]
be a maximally represented polynomial. Then $f(x)$ can be factored into
\[
f(x) = a_n \. x^{\. r} \.  (x \+ d_{r+1}) \. (x \+ d_{r+2}) \. \cdots \. (x \+ d_n),
\]
where $d_k = a_{k-1} - a_{k}$ for all $k = r+1, r+2, \ldots, n$. 
\end{thm}
\begin{proof}
Just expand the polynomial and we can get the conclusion of the theorem.
\end{proof}

The Tropical Fundamental Theorem of Algebra is stated in several places, for example~\cite{ss09}, and a proof (different from this paper) is given in~\cite{gm07}.

%%%
\section{Tropical Polynomials in Compact Forms}\label{s:compact}
We have seen that each tropical polynomial is equivalent to a unique maximally represented polynomial. In this section, we try to do something in opposite direction. We would like to find the most ``compact'' form of a given tropical polynomial. In Example~\ref{eg:differentforms}, we have seen that the coefficient $x$ term in the polynomial $(-1)\. x^{\. 2} \+ x\+ 1$ can be changed from $0$ to $-1$. We calm that if we can use a smaller coefficient for a specific term, we can actually drop that term (that is, set the coefficient to $-\infty$) without change the graph of the polynomial. For our case, $f(x) = (-1)\. x^{\. 2} \+ x\+ 1$ is equivalent to $(-1) \. x^{\.2} \+ 1$.

\begin{thm}\label{thm:drop}
Let $f(x) = a_n \. x^{\.n} \+ a_{n-1} \. x^{n-1} \+ \cdots \+ a_r \. x^r$ be a tropical polynomial. If there exist $b_k < a_k$ such that $g(x) = a_n \. x^{\.n} + a_{n-1} \. x^{n-1} \+ \cdots \+ a_{k-1} \. x^{\. (k-1)} \+ b_k \. x^{\. k} \+ a_{k+1} \. x^{\. k} \+ \cdots \+ a_r \. x^r$ and $f(x)$ define the same function, then we can drop the $x^{\.k}$ term completely. That is, set $b_k = -\infty$.
\end{thm}
\begin{proof}
Since
\[
f(x) \geq a_k + kx > b_k + kx,
\]
but $f(x)$ and $g(x)$ define the same function, so $b_k +kx$ is never been the maximum of the linear terms $a_i + ix$ for all $n \geq i \geq r$. Hence, we can complete drop the $x^k$ term without changing the graph of $f(x)$.
\end{proof}

We call a tropical polynomial is in its compact form if we drop all ``unnecessary terms.''

\begin{defn}\label{defn:compact}
Let $f(x) = a_n \. x^{\. n} \+ a_{n-1} \. x^{\. (n-1)} \+ \cdots \+ a_r \. x^r$ be a tropical polynomial, where $a_n, a_r \in \R$ and $a_i \in \T$ for all $r < i < n$. We say that $f(x)$ is in the compact form if for all $g(x)$ obtained from $f(x)$ by substitute a coefficient $a_i$ with $\alpha$, $r < i <n$, such that $\alpha < a_i$, then $g(x)$ and $f(x)$ define different functions.
\end{defn}

A tropical polynomial in its compact form is easy to calculate the roots, as the following Corollary shows.

\begin{cor}
Let
\[
f(x) = a_{n_1}\. x^{\. n_1} \+ a_{n_2} \. x^{\. n_2} \+ \cdots \+ a_{n_r}\. x^{\. n_r}
\]
be a tropical polynomial in its compact form, where $n_1 < n_2 < \cdots < n_r$ are all positive integers, and $a_{n_1}, a_{n_2}, \ldots, a_{n_r}$ in $\R$. Then $(a_{n_k} - a_{n_{k-1}})/(n_{k-1} - n_{k})$ is a root of $f(x)$ with multiplicity $n_{k-1} - n_{k}$, for all $k = 2, 3, \ldots, r$.
\end{cor}
\begin{proof}
For each $k=2, 3, \ldots, r$, $a_{n_{k-1}} + n_{k-1} x = a_k + n_k x$. Solve for $x$ we get $x = (a_{n_k} - a_{n_{k-1}})/(n_{k-1} - n_{k})$. The multiplicity is the change of the slope which equals to $n_k - n_{k-1}$.
\end{proof}

%%%
\section{Tropicalization of a Classical Polynomial}\label{s:tropicalization}
We will detour a bit to study two variable polynomials here for two reasons. First, we will see which part of this paper can be easily extended to multivariable cases. Second, the geometric pictures are more clear in two variable cases. Our purpose is just to give some motivations of studying tropical polynomials and more general, tropical meromorphic functions. Therefore, we will skip some details and proofs, please refer to~\cite{gathmann06, ims09, rst05} for more details.

\begin{defn}
For a tropical polynomial of two variables
\[
f(x,y) = \sum_{(i,j) \in I}^\+ a_{i,j} \. x^{\.i} \.  y^{\.j},
\]
where $a_{i,j} \in \R$ and $I$ is a finite subset of $\N^2$. Evaluate the tropical polynomial, we get
\[
f(x,y) = \max \{ a_{i,j} + i \cdot x + j \cdot y \}.
\]
We define the zero locus to be the points in $\R^2$ such that the maximum of the leaner forms is attained at least twice. The zero locus of the tropical polynomial will be denoted by $\TT(f)$, and we call it a tropical curve defined by $f(x,y)$.  
\end{defn}
Similarly, we can define tropical hypersurfaces in higher dimensional cases. We give an example of a tropical curve defined by a linear function which we call a tropical line.

\begin{eg}\label{eg:tropicalline}
Let $f(x,y) = x \+ y \+ 0$. The zero locus of $f(x,y)$ is shown in Figure~\ref{fig:tropicalline}, which we call a tropical line.
\end{eg}

\begin{figure}[h]
\begin{center}
\includegraphics[scale=0.7]{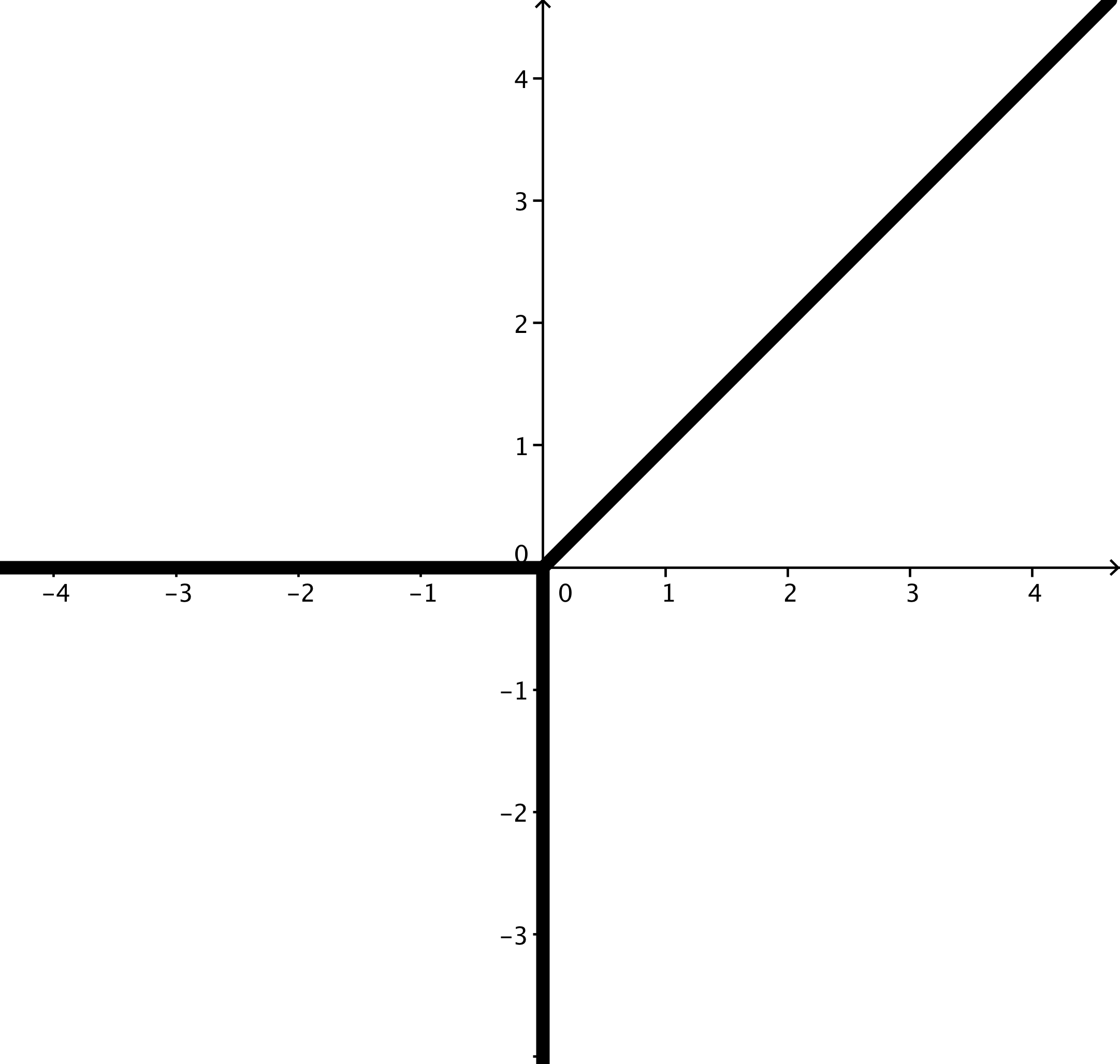}
\caption{The tropical line defined by $f(x,y) = x \+ y \+ 0$. }\label{fig:tropicalline}
\end{center}
\end{figure}

There are at least two ways to see how tropical curves (or hypersurfaces) comes from the classical curves. To explain this, suppose $V \subset \C^2$ be an algebraic variety defined by the locus of single complex polynomial $g(x,y)$. Consider the map
\[
\phi_t: (\C\setminus \{ 0 \})^2 \to \R^2
\]
defined by
\[
\phi_t(x,y) = (-\log_t(|x|), -\log_t(|y|)),
\]
where $t$ is a positive real number. Then we can ``see'' the complex curve $C$ in $\R^2$, namely, $\widetilde{V} = \phi((\C\setminus \{ 0 \})^2 \cap C)$. When we take $t$ approaches to zero, $\widetilde{V}$ will be a tropical curve. 

\begin{eg}
Let $C$ be a complex line defined by the polynomial $f(x,y) = x + y +1$. Figure~\ref{fig:amoeba} shows that as $t$ approaches to zero, the image $\phi_t(C)$ tends to the tropical line defined in Example~\ref{eg:tropicalline}.
\end{eg}

\begin{figure}[htbp]
\begin{center}
\subfigure[$\phi_t(C), t = 0.4.$]{\includegraphics[scale=0.4]{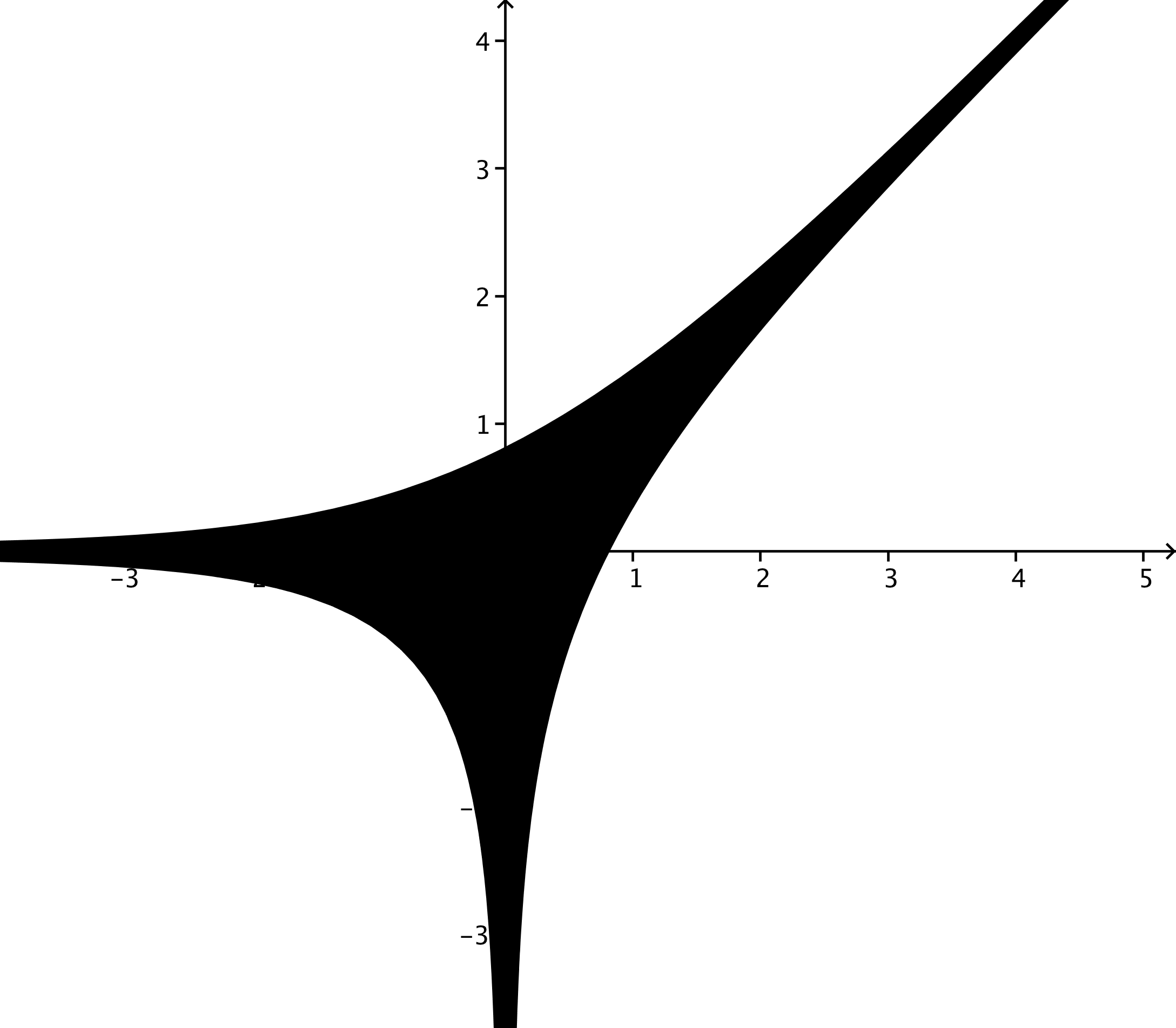}}
\subfigure[$\phi_t(C), t = 0.1.$]{\includegraphics[scale=0.4]{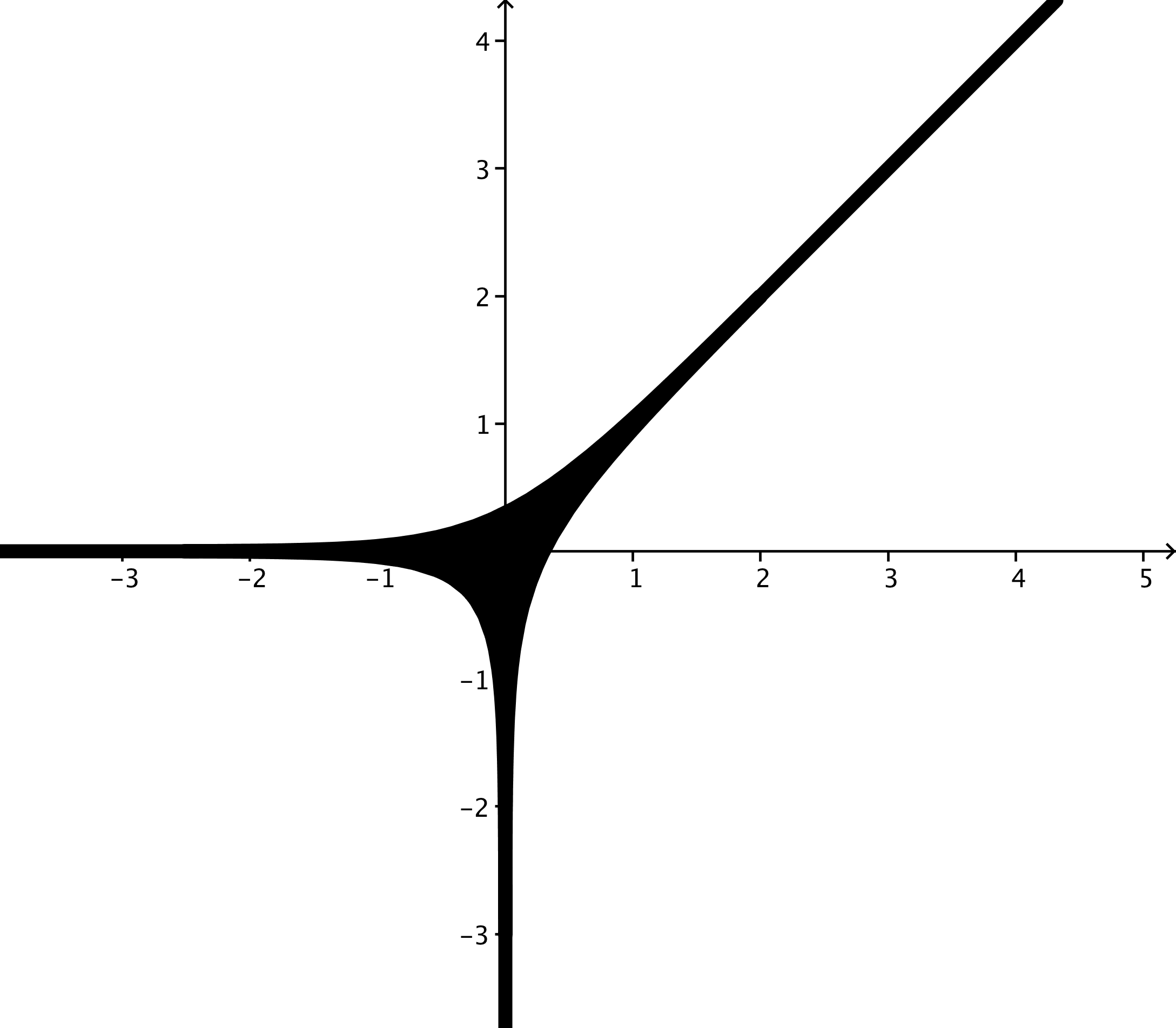}} 
\subfigure[$\lim_{t \to 0} \phi_t(C)$, the tropical line $x + y +0$.]{\includegraphics[scale=0.4]{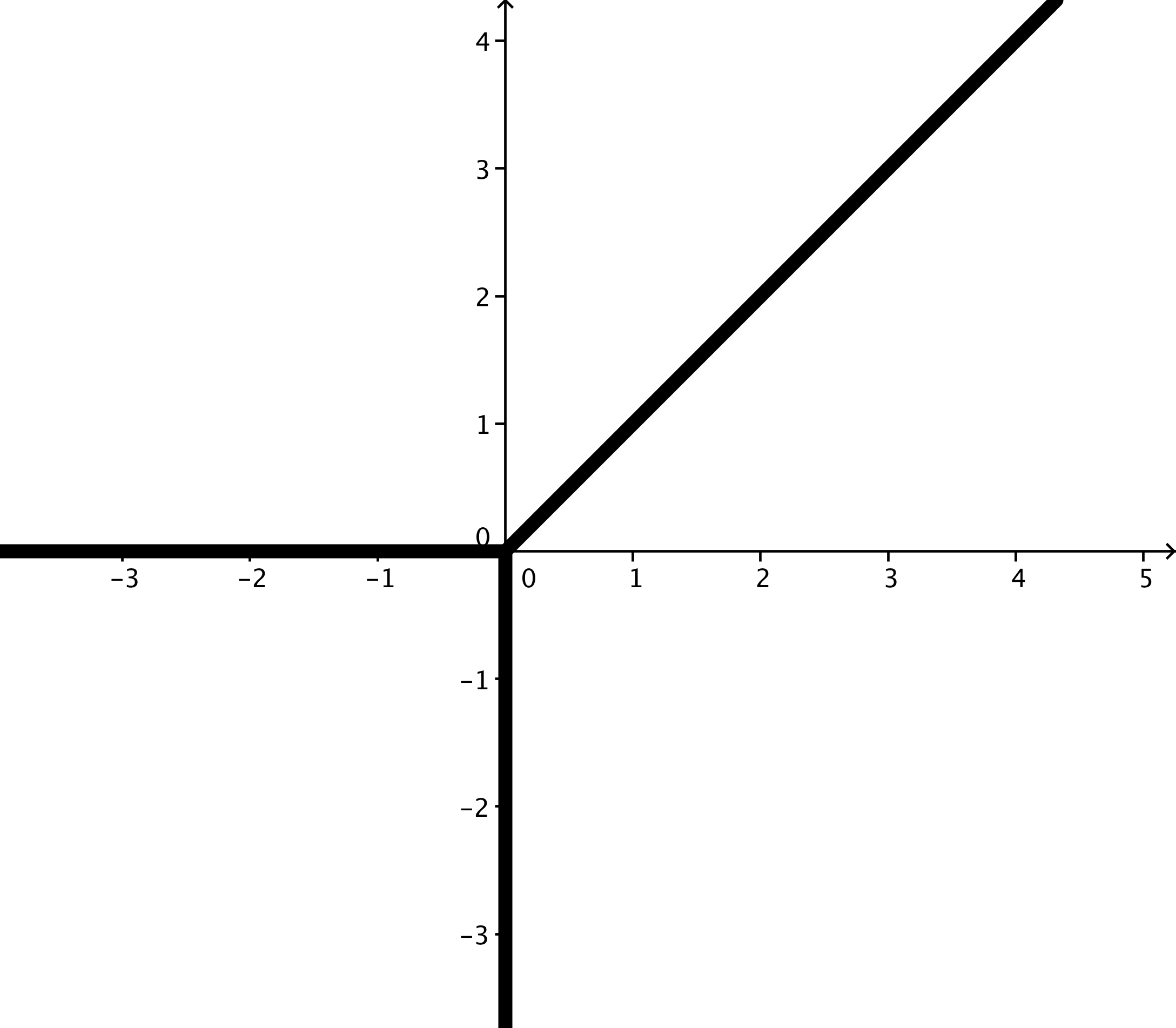}}
\caption{Amoeba corresponding to $f(x,y) = x + y +1$.}\label{fig:amoeba}
\end{center}
\end{figure}

We can do it without the ambiguous taking limit part. Let 
\[
K = \cup_{n=1}^{\infty} ((t^{1/n}))
\]
be the field of Puiseux series. It is well known that $K$ is an algebraically closed field, so we can do classical algebraic geometry with this field $K$. An element $a \in K$ is a power series of $t$, and the exponents can be rational, but bounded below, and have the same denominator. 

We define a valuation $\val$ on $K$ by set $\val(a)$ to be the minimum exponent appears in $a \in K$, and $\val(0) = -\infty$. Define a map
\[
\Val \colon K^2 \to \T^2
\]
by 
\[
\Val(x,y) = (-\val(x), -\val(y)).
\]

Let 
\[
f(x,y) = \sum_{(i,j) \in I} a_{i,j} x^i y^j
\]
be a polynomial in $K[x,y]$. Let $C$ be the curve define by $f(x,y)$. We denote the image of $C$ under the map $\Val$ by $\AA(C)$ and call it a non-archimedean amoeba corresponding to the curve $C$. 

The tropical polynomial corresponding to $f(x,y)$ is defined by
\[
\sum_{(i,j) \in I}^\+ \val(a_{i,j}) \. \val(x)^{\. i} \. \val(y)^{\. j},
\]
and we will abuse the notation a bit to write
\[
g(x,y) = \sum_{(i,j) \in I}^\+ \val(a_{i,j}) \. x^{\. i} \.  y^{\. j}.
\]
We call $g(x,y)$ the \emph{trpoicalization} of the polynomial $f(x,y)$ and denoted $g(x,y)$ by $\T(f)$. It is not completely clear that $\AA(C) = \TT(\T(f))$, but a theorem of Kapranov~\cite{gkz94} assures it is the case.

\begin{thm}[Kapranov]
\[
\AA(C) = \TT(\T(f)).
\]
\end{thm}

\begin{eg}
Consider $f(x,y) = x + y + 1$ be a polynomial in $K[x,y]$. Then $\T(f) = \val(1)\. x \+ \val(1)\. y   \+ \val(1) = x + y + 0$. Hence, $\TT(\T(f))$ is exactly the tropical line in Example~\ref{eg:tropicalline}.
\end{eg}

\section{Tropical Meromorphic Functions with Prescribed Roots and Poles}\label{s:prescribed}
Given a finite subset $\{d_1, d_2, \ldots, d_r\} \subset \T$, $\{m_1, m_2, \ldots, m_r\} \subset \N^r$, can we find a polynomial whose roots are the $d_1, d_2, \ldots, d_r$ with given multiplicities $m_1, m_2, \ldots, m_r$, respectively? It is clear that
\[
f(x) = K\. (x \+ d_1)^{\. m_1} \. (x \+ d_2)^{\. m_2} \. \cdots \.  (x \+ d_r)^{\. m_r},
\]
for some constant $K$, satisfies the conditions. We claim that these are the complete solutions.

\begin{thm}\label{thm:givenzeros}
Let $f(x)$ be a tropical meromorphic function. Then $f(x)$ has finitely many roots with no poles if and only if $f(x)$ is a tropical polynomial, and there exist $K \in \R$, $d_1 < d_2 < \ldots < d_r \in \T$ and $m_1, m_2, \ldots, m_r \in \N$, such that
\begin{align}
f(x) 	&= K \. (x \+ d_1)^{\. m_1} \. (x \+ d_2)^{\. m_2} \. \cdots \. (x \+ d_r)^{\. m_r} \label{eq:givenroots1}\\
	&= K \. [x^{\. n} \+ \sum^{\+}_{1 \leq i \leq r}\sum^{\+}_{1 \leq j \leq m_i} (m_1 d_1 + m_2 d_2 + \cdots + m_{i-1} d_{i-1} + j d_i) \\
	&\  \. x^{\. (n - m_1 - m_2 - \ldots - m_{i-1} - j)}] \label{eq:givenzeros2} \notag
\end{align}
where $n = m_1 + m_2 + \cdots + m_r$.
\end{thm}
\begin{proof}
If $f(x)$ is a tropical polynomial, then clearly $f(x)$ has finitely many roots and no poles. Conversely, suppose that $f(x)$ has finitely many roots with no poles. Let $d_1 < d_2 < \ldots < d_r \in \T$ be the poles of $f(x) $ and  $m_1, m_2, \ldots, m_r \in \N$ be the corresponding multiplicities. We will prove the case $d_1 \neq -\infty$, and the case $d_1 = - \infty$ can be proved by the similar arguments. By assumption $-\infty$ is not a root, so there when $x < d_1$, $f(x)$ is a constant function. Pick any $x_0 < d_1$. Define a tropical polynomial
\[
g(x) = K \. (x \+ d_1)^{\. m_1} \. (x \+ d_2)^{\. m_2} \. \cdots \. (x \+ d_r)^{\. m_r},
\]
where $K = f(x_0) - m_1 d_1 - m_2 d_2 - \cdots - m_r d_r$. We claim $f \equiv g$ as a function. Let $x$ be an arbitrary real number. Let
\[
k = \max \{s | x > d_s \}.
\]
Evaluate $f(x)$, we get
\begin{equation}\label{eq:evfx}
\begin{split}
f(x) 	= &f(x_0) + (d_2 - d_1) m_1 + (d_3 - d_2) (m_1 + m_2) + \cdots \\
	    &+ (d_k - d_{k-1}) (m_1 + m_2 + \cdots + m_k) \\
	    & + (x-d_k) (m_1 + m_2 + \cdots + m_k)\\
	= &f(x_0) - m_1 d_1 - m_2 d_2 - \cdots - m_{k} d_k + (m_1 + m_2 + \cdots + m_k) x.
\end{split}
\end{equation}

Evaluate $g(x)$, we obtain
\begin{align*}
g(x) 	&= K + m_1 x + m_2 x + \cdots + m_k x + m _{k+1} d_{k+1} + \cdots + m_r dr \\
	&= f(x_0) - m_1 d_1 - m_2 d_2 - \cdots - m_k d_k + (m_1 + m_2 + \cdots + m_k) x.
\end{align*}

Therefore, $f(x) = g(x)$, and expanding $g(x)$ we get the equation~\ref{eq:givenzeros2}, so we prove the theorem. 
\end{proof}

We apply Theorem~\ref{thm:givenzeros} to give an algorithm to find a tropical polynomial with prescribed roots. Let $d_1, d_2, \ldots, d_r$ be roots of multiplicities $m_1, m_2, \ldots, m_r$, respectively. Repeat a root as many times as its multiplicity, we get a new sequence of roots $b_n \geq b_{n-1} \geq \cdots \geq b_1$, where $n=m_1 + m_2 + \cdots + m_r$. Then one possible polynomial $f(x)$ with these roots is:
\[
x^n + a_{n-1} x^{n-1} + \cdots + a_0,
\]
where 
\[
a_{n-1} = b_{n},
\]
and
\[
a_i = b_{i+1} + a_{i+1} \mbox{, for all $i = n-2, n-3, \ldots, 0$}.
\]

\begin{eg}\label{eg:prezeros}
Let $f(x)$ be a polynomial with roots $1, 2, 3$ of multiplicities 1. Applying the above algorithm, one solution is the following polynomial:
\[
f(x) = x^{\.3} \+ a_2 \. x^{\. 2} \+ a_1 \. x \+ a_0,
\]
where $a_2 = 3$, $a_1 = 2 + a_2 = 5$, $a_0 = 1 + a_1 = 6$. Hence, 
\[
f(x) = x^{\.3} \+ 3\. x^{\. 2} \+ 5 \. x \+ 6.
\]
The graph of $f(x)$ is shown in Figure~\ref{fig:polytrans-a}. When we tropically multiply $f(x)$ by any nonconstant number $k$, we will move the graph up or down by $|k|$ and the roots are the same. For instance, let $g(x) = 2\. f(x) = 2 + f(x)$. The polynomial $g(x)$ has three roots $1,2$, and $3$. The graph of $g(x)$ is the graph of $f(x)$ move up by $2$, as the dotted lines in Figure~\ref{fig:polytrans-b}.

\begin{figure}[h]
\begin{center}
\subfigure[The graph of $f(x) = x^{\.3} \+ 3\. x^{\. 2} \+ 5 \. x \+ 6$.]{\label{fig:polytrans-a}\includegraphics[scale=0.7]{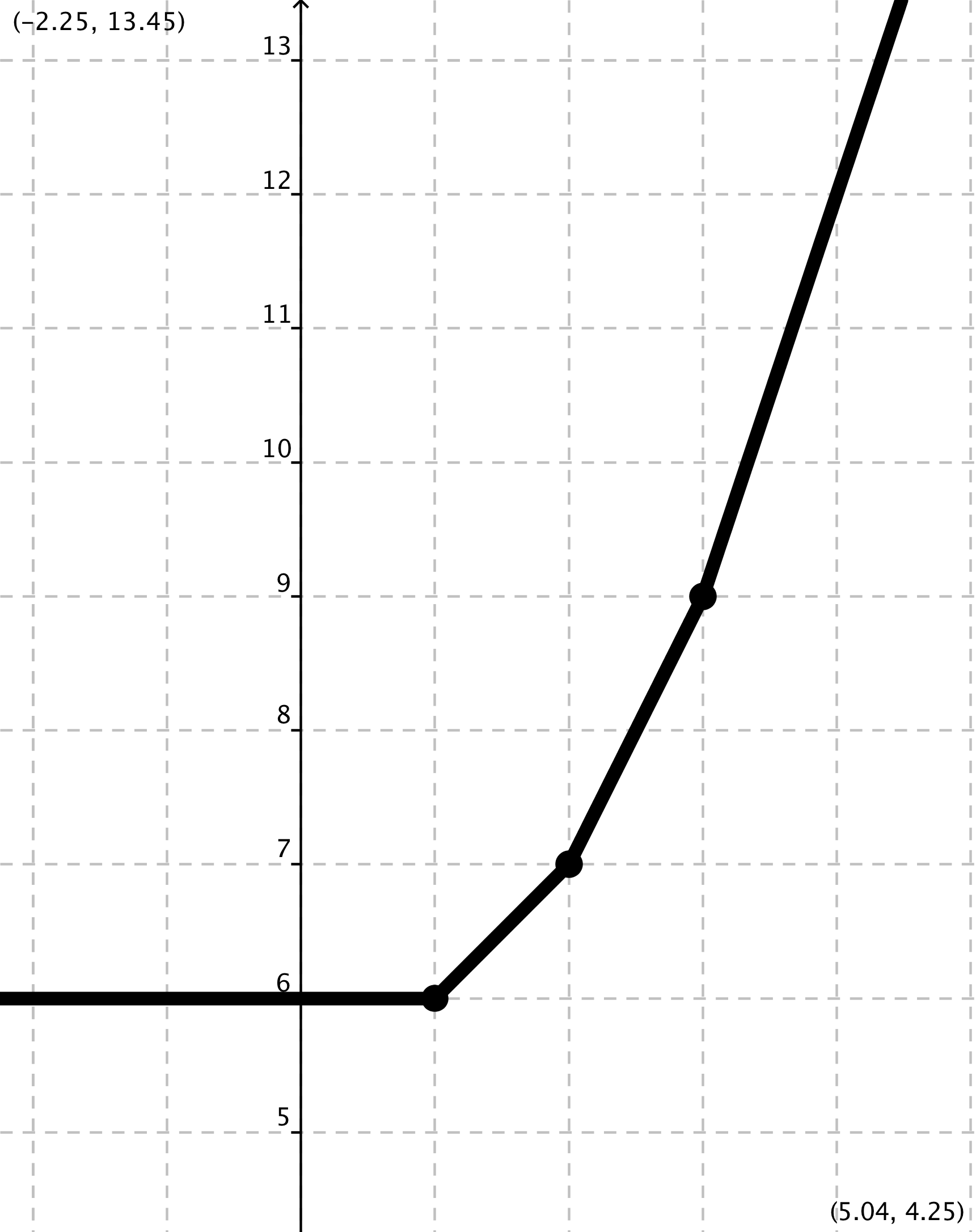}}
\subfigure[The dotted line is the graph of $2\. f(x)$.]{\label{fig:polytrans-b}\includegraphics[scale=0.7]{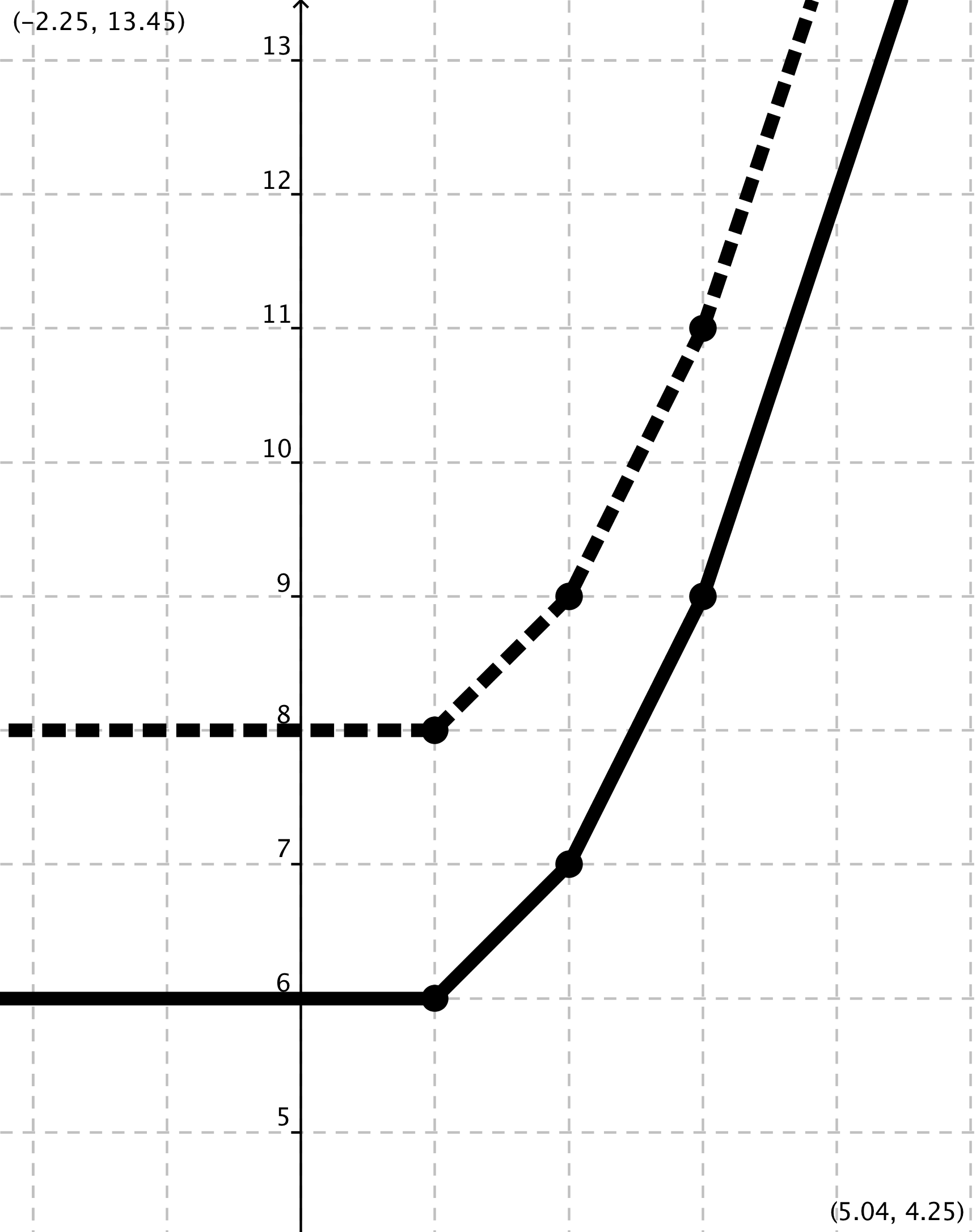}}
\end{center}
\caption{Compare the graph of the polynomial $f(x) = x^{\.3} \+ 3\. x^{\. 2} \+ 5 \. x \+ 6$ and $2 \. f(x)$.}\label{fig:polytrans}
\end{figure}
\end{eg}

\begin{cor}
 Given $d_1 < d_2 < \cdots <d_r \in \T$, and $m_1, m_2, \ldots, m_r$. There exists a tropical polynomial with roots $d_1, d_2, \ldots, d_r$ of multiplicities $m_1, m_2, \ldots, m_r$, respectively. Moreover, the tropical polynomial is unique (as a function) up to constant multiples.
\end{cor}

We explain the key point of the proof of the Thereom~\ref{thm:givenzeros} a bit more and try to generalize the results to tropical rational and meromorphic functions. Suppose $d_1 < d_2 < \cdots < d_r$ are some real roots with certain multiplicities $m_1, m_2, \ldots, m_r$. The left most part of the graph of the function is either with slope zero ($-\infty$ is not a root) or some positive integer $m$ ($-\infty$ is a root). In either case, we know there is a $x_0 < d_1$, and $(x_0, f(x_0))$ is on a line with slope either $m\geq 0$. Therefore, we can find the exact equation of the line. Then, $x$ moves along to the point $d_1$. The ``speed'' increases by $m_1$ and hence we can once again find the exact formula for the line. It is easy to see, the whole graph of $f(x)$ is determined once we get $(x_0, f(x_0))$ and the slope at that beginning point. 

Hence, very similar to the polynomial case, we have the following theorem for tropical rational functions.

\begin{thm}\label{thm:givenrational}
Let $f(x)$ be a tropical meromorphic function. Suppose that $f(x)$ has finitely many roots and poles if and only if $f(x)$ is a tropical rational function. Moreover, if $b_1, b_2, \ldots, b_r$ are roots of multiplicities $m_1, m_2, \ldots, m_r$ and $c_1, c_2, \ldots, c_s$ are poles of multiplicities $n_1, n_2, \ldots, n_s$, respectively, then the tropical rational function is 
\[
f(x) = g(x) \/ h(x),
\]
where
\[
g(x) = K \. (x \+ b_1)^{\. m_1} \. (x \+ b_2)^{\. m_2} \. \cdots \. (x \+ b_r)^{\. m_r},
\]
and
\[
h(x) = (x \+ c_1)^{\. n_1} \. (x \+ c_2)^{\. n_2} \. \cdots \. (x \+ c_s)^{\. n_s},
\]
for some $K \in \R$.
\end{thm}

Finally, given a prescribed roots and poles, there is a unique tropical meromorphic function up to constant multiplicities.

\begin{thm}\label{thm:givenmero}
Let $\mathcal{Z} = \{b_i\}_{i \in I} \subset \T$ and $\mathcal{P} = \{c_j \}_{j \in J} \subset \T$, where $I$ and $J$ are at most countable index sets,  such that $\mathcal{Z} \cap \mathcal{P} = \emptyset$. Let $\{ m_i \}_{i \in I}$ and $\{ n_j \}_{j \in J}$ are collections of positive integers. Then there exists a unique tropical meromorphic function $f(x)$ (up to constant multiplies) such that $f(x)$ has roots $\mathcal{Z}$ of multiplicities $\{ m_i \}_I$ and pols $\mathcal{P}$ of multiplicities $\{ n_j \}_J$, respectively.
\end{thm}

%%%
\section{Extended Tropical Meromorphic Functions}
In this section, we elaborate some subtle different definitions of tropical meromorphic functions on recently researches. In~\cite{lt09}, a tropical meromorphic function is allowed to have real slopes, we call this kind of tropical meromorphic function an extended tropical meromorphic function.

\begin{defn}
An extended tropical polynomial is of the form:
\[
f(x) = a_n \. x^{r_n} \+ a_{n-1} \. x^{r_{n-1}} \+ a_1 \. x^{r_1},
\]
where $r_i \in \R^+$ for all $i = 1, 2, \ldots, n$.
\end{defn}

The definition of roots, poles, and multiplicities of tropical meromorphic functions can still apply to extended ones. Extended tropical meromorphic functions are not directly from the classical geometry as we explained in Section~\ref{s:tropicalization}. However, these functions arise naturally in many real world problems, therefore have many possible applications. 

Recall that an $\R$-tropical meromorphic function $f(x)$ is a piecewise linear function with integer slopes, and it is different from a tropical meromorphic function just because $-\infty$ is not in the domain. Hence, we do not consider $-\infty$ as either a root or a pole for any  $\R$-tropical meromorphic function. 
% on R and on T
% on R: can extend to a_-1 a_-2 , ... 週期性
% on R: zero and pole up to
\begin{eg}
As in Example~\ref{eg:prezeros}, let $f(x) = x^{\.3} \+ 3\. x^{\. 2} \+ 5 \. x \+ 6$ be an $\R$-tropical polynomial with roots $1, 2, 3$ of multiplicities $1$. Let $g(x) = x\. f(x)$ be another $\R$-tropical polynomial, where $f(x)$ and $g(x)$ are not a tropical multiple of each other, so Theorem~\\ref{thm:givenzeros} fails. The theorem fails because $g(x)$ actually has one more root than $f(x)$ has, namely $-\infty$, which we do not count in $\R$-tropical meromorphic functions. The graphs are as in Figure~\ref{fig:xfx}. 

\begin{figure}
\begin{center}
\includegraphics[scale=0.7]{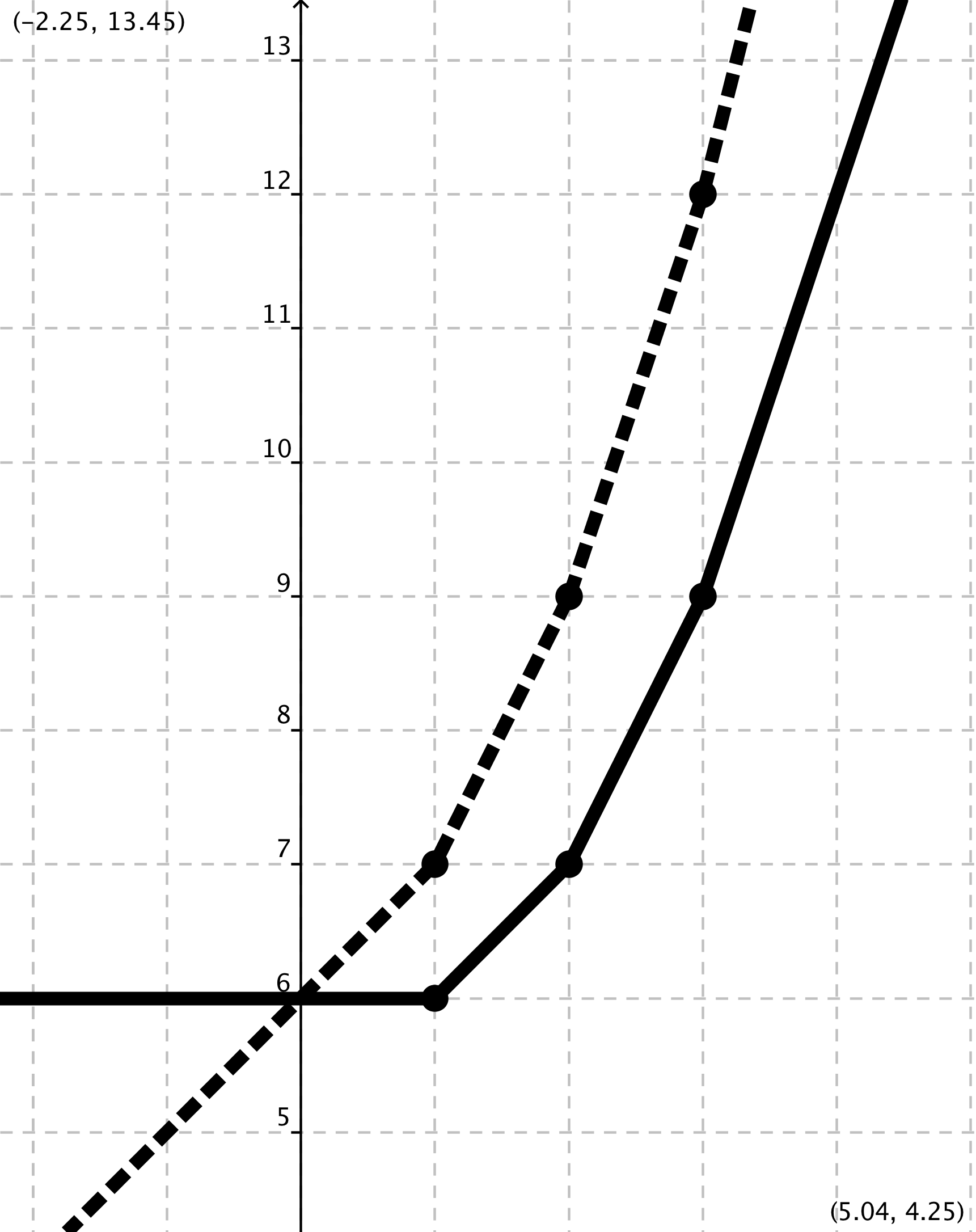}
\end{center}
\caption{Compare $f(x) =  x^{\.3} \+ 3\. x^{\. 2} \+ 5 \. x \+ 6$ with $x \cdot f(x)$ (dotted line).}\label{fig:xfx}
\end{figure}
\end{eg}

Theorems~\ref{thm:givenzeros}, \ref{thm:givenrational},  \ref{thm:givenmero} will hold for a slightly modification. With prescribed roots and poles, $f(x)$ and $g(x)$ are unique up to a linear term. That is 
\[
f(x) = g(x) + mx + b = b\.x^{\.m} \. g(x),
\]
for some $m \in \Z$ and $b \in \R$.

$\R$-tropical meromorphic functions some times are more general as the following example shows.

\begin{eg}\label{eg:period}
The function in Figure~\ref{fig:period} is a period piecewise linear function of all integers slopes on $\R$ is $\R$-tropical meromorphic, but not tropical meromorphic.

\begin{figure}[h]
\begin{center}
\includegraphics[scale=0.7]{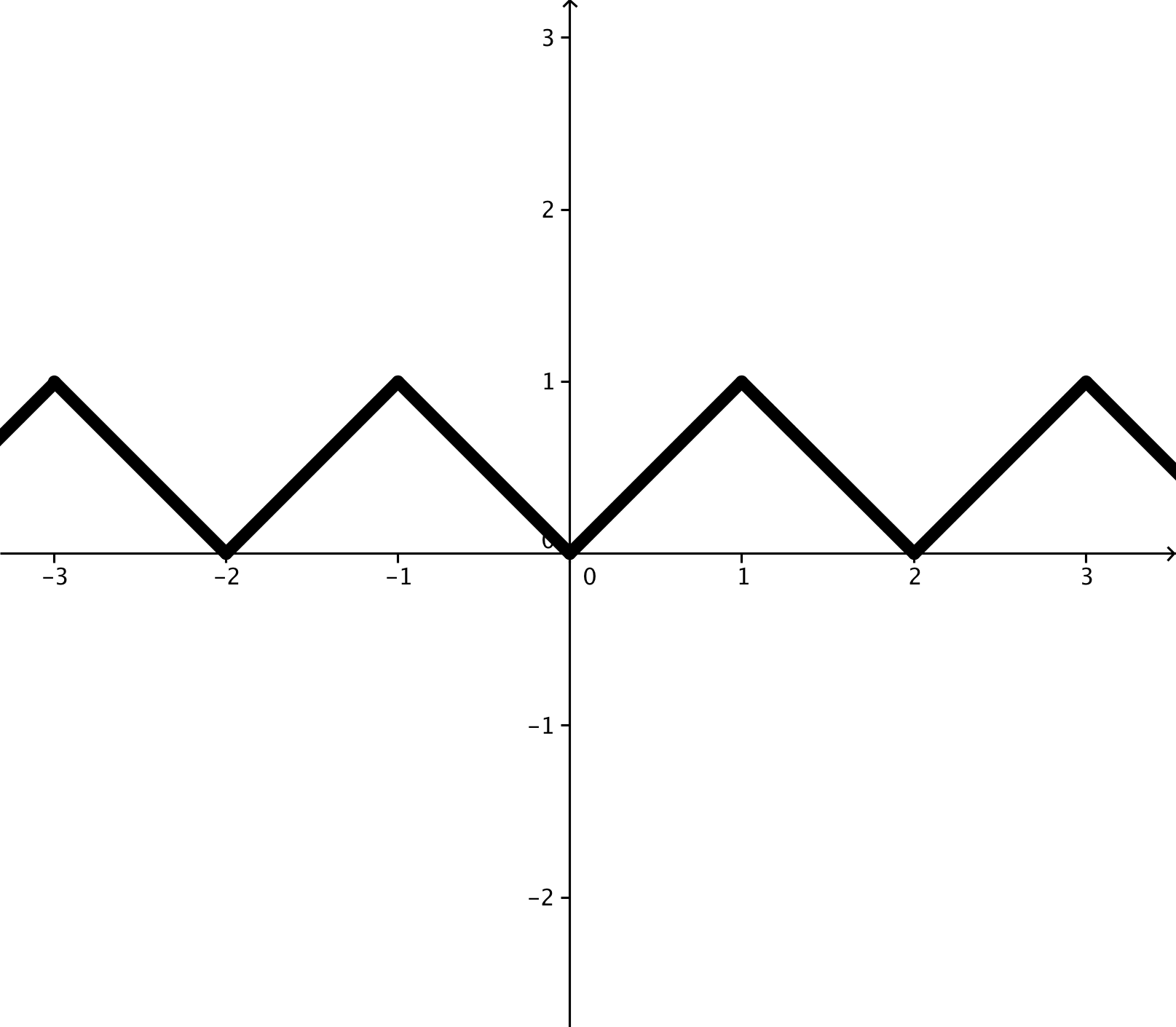}
\caption{A period piecewise linear function of all integer slopes is $\R$-tropical meromorphic.}
\label{fig:period}
\end{center}
\end{figure}
\end{eg}

On the other hand, our definition of tropical meromorphic functions is more naturally from classical geometry. Many properties are more reasonable in some sense.

\begin{eg}
It is interesting that $f(x) = |x|$ (Figure~\ref{fig:abs}) can be treated as a tropical rational function. It has a root of multiplicity $2$ at $x=0$. If we consider $f(x)$ a $\R$-tropical meromorphic function, we would think it is a polynomial since there is no pole. However, $f(x)$ is not a tropical polynomial. It has $-\infty$ as a pole of multiplicity $1$. Thereofre,
\[
f(x) = (x \+ 0)^{\. 2} \/ x.
\]

\begin{figure}[h]
\begin{center}
\includegraphics[scale=0.7]{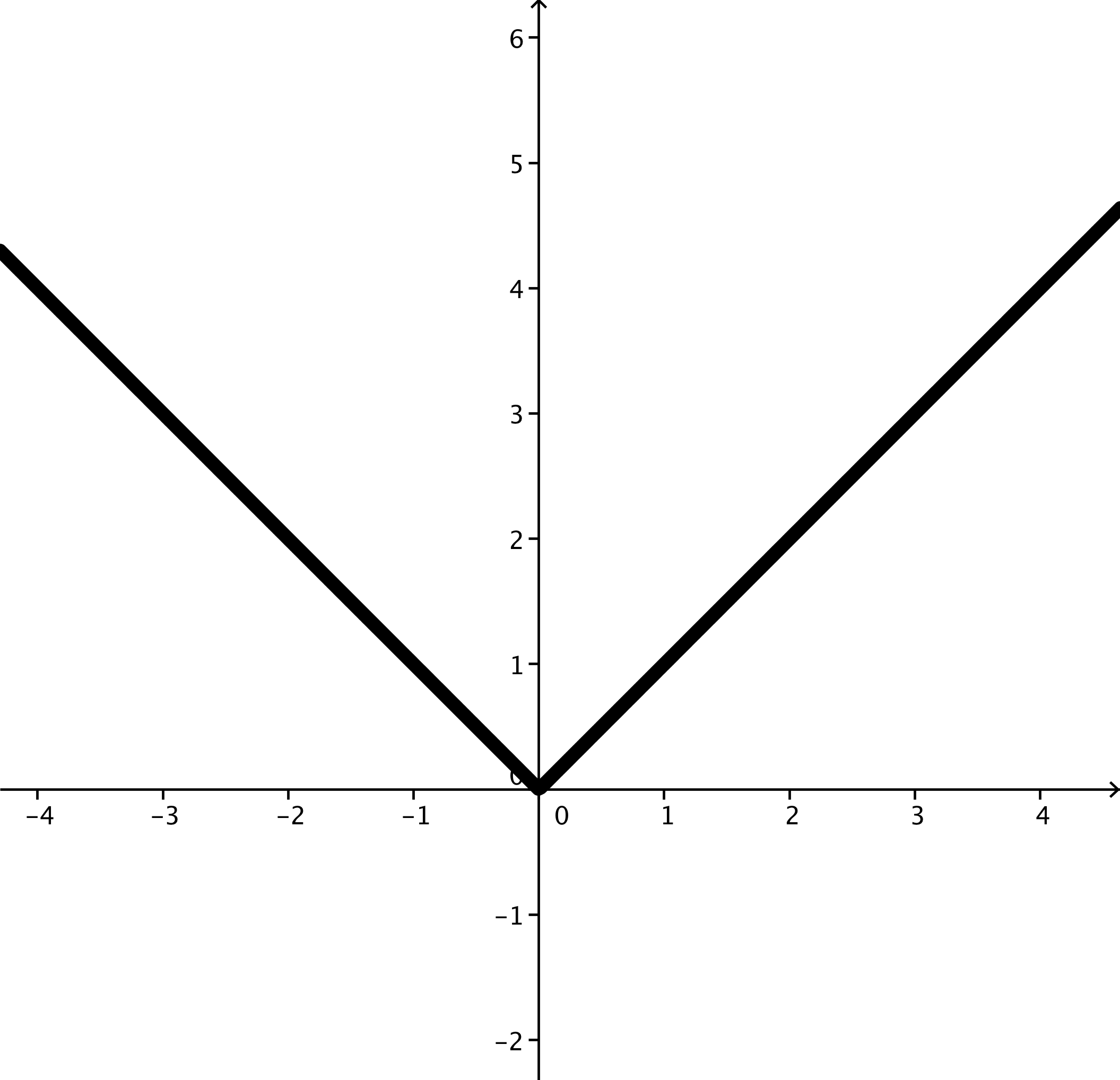}
\caption{$f(x) = |x|$ is a tropical rational function $(x \+ 0)^{\. 2} \/ x$.}
\label{fig:abs}
\end{center}
\end{figure}
\end{eg}

\section{More Theorems Related to Complex Analysis}
In complex analysis, Liouville's theorem says that if $f(x)$ is  a bounded entire function then $f(x)$ is constant. The tropical version of Liouville's theorem is trivial.

\begin{thm}[Tropical Liouville's Theorem]
Let $f(x)$ be a tropical entire function (a tropical meromorphic function with no pole). If $f(x)$ is bounded then $f(x)$ is constant.
\end{thm}

The maximum modulus theorem in complex analysis says that if $G$ is a region of $\C$ and $f(x)$ is a complex analytic function such that there is a point $z$ in $G$ such that $|f(z)| > |f(x)|$ for all $x \in G$. Then $f(x)$ is constant. We do not have exactly the same maximum modulus theorem, but the fact that both maximum and minimum of a tropical function with no poles (even locally) appear at end points is trivial.

\begin{thm}[Tropical Maximum Modulus Theorem]
Let $I = [a,b]$ be a closed interval of $\R$. Let $f(x)$ be a tropical meromorphic function and there is a tropical entire function $F(x)$ such that $F|_I (x) = f|_I (x) $.  Then $\max \{ f(x) | a \leq x \leq  b \} = f(b)$ and $\min \{ f(x) | a \leq x \leq b \} = f(a)$.
\end{thm}

\bibliographystyle{hplain}
\bibliography{reference}
\end{document}